\documentclass{amsart}
\usepackage{hyperref}
\hypersetup{
    colorlinks,%
    citecolor=blue,%
    filecolor=blue,%
    linkcolor=blue,%
    urlcolor=blue
}
\usepackage{enumerate}
\usepackage[all]{xy}
\usepackage{todonotes}
\usepackage{cases}
\usepackage{tikz}
\usepackage[scriptsize]{caption}

\usepackage[backrefs,lite]{amsrefs}

\newtheorem{thm}{Theorem}[section]

\newtheorem{prop}[thm]{Proposition}

\theoremstyle{definition}

\newtheorem{defn}[thm]{Definition}
\newtheorem*{defn*}{Definition}
\newtheorem{ex}[thm]{Example}
\newtheorem{rem}[thm]{Remark}

\newcommand{\A}{\mathbb{A}}
\newcommand{\F}{\mathbb{F}}
\newcommand{\Q}{\mathbb{Q}}
\newcommand{\qq}{\mathbb{Q}}

\newcommand{\pp}{\mathbb{P}}
\newcommand{\kk}{\mathbb{K}}
\newcommand{\zz}{\mathbb{Z}}

\newcommand{\oo}{\mathcal{O}}
\newcommand{\dd}{\mathcal{D}}
\newcommand{\Ss}{\mathcal{S}}
\newcommand{\rr}{\mathcal{R}}
\newcommand{\xx}{\mathcal{X}}

\newcommand{\Cl}{\operatorname{Cl}}

\numberwithin{equation}{section}

\begin{document}

\title{On deformations of toric varieties}

\author[A.~Laface]{Antonio Laface}
\address{
Departamento de Matem\'atica,
Universidad de Concepci\'on,
Casilla 160-C,
Concepci\'on, Chile}
\email{alaface@udec.cl}

\author[M.~Melo]{Manuel Melo}
\address{
Departamento de Matem\'atica,
Universidad de Concepci\'on,
Casilla 160-C,
Concepci\'on, Chile}
\email{manuelmelo@udec.cl}

\subjclass[2010]{
14D15, 14M25.\\
Both authors have been partially supported 
by Proyecto FONDECYT Regular N. 1150732,
Proyecto Conicyt/DAAD PCCI13005 and 
Proyeto Anillo ACT 1415 PIA Conicyt.
}

\maketitle

\begin{abstract}
Let $X$ be a smooth complete toric variety. We describe
the Altmann-Ilten-Vollmert equivariant deformations of
toric varieties in the language of Cox rings. More
precisely we construct 
one parameters families $\pi\colon\mathcal X\to\mathbb{A}^1$ 
of deformations of $X$, such that the total space $\mathcal X$ 
of the deformation is a $T$-variety of complexity one, defined
by a trinomial equation, and the map $\pi$ is equivariant with 
respect to the torus action. Moreover we show that the images
of all these families via the Kodaira-Spencer map form a basis
of the vector space $H^1(X,T_X)$.
\end{abstract}

\section*{Introduction}
The topic of deformations of toric varieties
has been studied by K. Altmann in \cite{Alt} and
A. Mavlyutov in \cite{Mav}.
In the affine case they describe toric 
deformations in a combinatorial way via
polyhedral decompositions of linear sections
of the defining cone of the toric variety.
The theory of polyhedral divisors
is later developed in \cite{AH} and \cite{AHS}
as a generalization of toric varieties 
to T-varieties, i.e. varieties coming 
with a torus action.
N. Ilten and R. Vollmert 
make use of the language
of polyhedral divisors
in \cite{IV} to
describe deformations of
T-varieties of complexity one.
Their method involves decomposing the
polyhedral data of the varieties,
similar to what Altmann
did for the affine case.
Moreover, in the case of smooth toric
varieties they also prove that
such deformations are in correspondence
with a generating set of the
space of infinitesimal deformations
of the starting variety.
In \cite{HS}, we are presented
with an explicit way to compute
the Cox ring of a T-variety,
starting from its polyhedral representation.
This serves as the main connection
between the work
of N. Ilten and R. Vollmert and
the present paper.

This paper is devoted to studying
deformations of smooth toric varieties
from the point of view of Cox rings,
in a similar spirit of \cite{Mav}.
Starting from a toric variety $X$ and
some extra combinatorial data,
we describe the Cox ring of a complexity
one variety $\xx$ which fits into
to a one-parameter
deformation $\xx\rightarrow\A^1$
of $X$, as shown in Theorem~\ref{one-param}.
After that, we proceed to study the
corresponding Kodaira-Spencer
map in Theorem \ref{ks}

The variety $\xx$ turns out to be the 
same variety introduced in~\cite{IV} and described
with the language of
polyhedral divisors.
Moreover, and much like
what was done for polyhedral divisors,
we show that the deformations we describe
generate the space of infinitesimal
deformations of $X$.
As applications of this theory we study
deformations of scrolls and deformations
of hypersurfaces of smooth toric varieties.

The paper is organized as follows:
Section 1 of this work covers some concepts
and properties of deformation theory
and toric geometry
which will be used thorughout the
rest of the article.
Section 2 is the central section of this paper,
where we explain
how to construct deformations
of a toric variety $X$,
as well as the fact that the images of
these deformations,
under the
Kodaira-Spencer map,
generate H$^1(X,T_X)$.
Section 3 is a summary on the language of
polyhedral divisors that was mentioned above.
We also mention Ilten's and Vollmert's way
of finding deformations of T-varieties
and note that these deformations are
equivalent the ones we give in the previous section.
Lastly, in Section 4, we apply our
results to the study of rational scrolls
over $\pp^1$,
finding exactly those which are rigid
and proving that every scroll
deforms to a rigid one, and to 
deformations of hypersurfaces
of toric varieties.

{\em Acknowledgments}. It is a pleasure to thank
Klaus Altmann, Nathan Ilten and Robin Guilbot
for useful discussions on the subject.

\section{Preliminaries}
\label{sec:1}
We recall here some basic facts about deformation theory
and toric varieties.
For the rest of this article, $\kk$ will be
an algebraically closed field of characteristic 0.
\subsection{Generalities on deformations}
\begin{defn}
Let $X$ be a scheme
over $\kk$.
A {\it deformation} of $X$ 
over a scheme $S$
is a flat surjective morphism of schemes $\pi:\mathcal{X}\rightarrow S$ that fits in
a cartesian diagram

$$\xymatrix@R=20pt@C=20pt{X=
{\rm Spec}(\kk)\times_S\xx
\ar[d] \ar@{->}[rr] && \mathcal{X}\ar[d]^{\pi}\\
{\rm Spec}(\kk) \ar[rr]^{s}&&S}$$

If $S$ is algebraic,
then for each rational point $t\in S$,
the scheme-theoretic fiber $\xx(t)$
is also called a {\it deformation} of $X$.
Given another deformation $\pi'\colon\xx'\rightarrow S$
of $X$, we say
that $\pi$ and $\pi'$ are {\it isomorphic} if there is a morphism
$\phi:\xx\rightarrow\xx'$ inducing the identity over $X$ and
such that the diagram

 \[
 \xymatrix@R=10pt@C=10pt{
 & \xx\ar[dd]|!{[d];[d]}\hole^(.35){\pi}
     \ar[rd]^\phi & \\
 X\ar[rr]\ar[dd]\ar[ur]
 & &\xx'\ar[dd]^(.35){\pi'}\\
 & S\ar@{=}[rd]
 & \\
 {\rm Spec}(\kk)\ar[rr]\ar[ur]&& S
 }
 \]
is commutative.
We denote the set of isomorphism classes
of deformations
by ${\rm Def}_X(S)$.
\end{defn}

\begin{thm}\label{kscorr}\cite[Theorem 2.4.1]{Ser}
If $X$ is a smooth scheme,
there is an isomorphism of vector spaces
$$\kappa: {\rm Def}_X(\kk[t]/\langle t^2\rangle)\stackrel{\sim}\longrightarrow {\rm H}^1(X,T_X),$$
where $T_X$ is the tangent sheaf of $X$.
\end{thm}

Let $X$ be a smooth algebraic variety
and consider a deformation
$\pi\colon \xx\rightarrow S$ of $X$.
Giving $\varphi\in T_{S,s}$
is equivalent to giving a morphism
$\varphi\colon{\rm Spec}\left(\kk[t]/\langle t^2\rangle\right)
\rightarrow S$ with image $s$.
Pulling back the deformation
by $\varphi$, we obtain
a deformation

$$\xymatrix@R=20pt@C=20pt{X
\ar[d] \ar@{->}[rr] && \xx\times_S{\rm Spec}\left(\kk[t]/\langle t^2\rangle\right)\ar[d]^{\pi'}\\
{\rm Spec}(\kk) \ar[rr]^{s}&&
{\rm Spec}\left(\kk[t]/\langle t^2\rangle\right)}$$

which, by Theorem \ref{kscorr}, corresponds
to an element of ${\rm H}^1(X,T_X)$.
\begin{defn}
The above construction gives a map
$$T_{S,s}\rightarrow {\rm H}^1(X,T_X)$$
called the \textit{Kodaira-Spencer map}
of the deformation $\pi$.

Note that in the case $S={\rm Spec}\,\kk[x]$,
the morphism $\varphi$ above is
uniquely defined up to scalar multiplication,
so the image of the Kodaira-Spencer map
is determined by a single element in H$^1(X,T_X)$.
\end{defn}

\subsection{Toric varieties}
A variety $X$ is called a {\it toric variety} if
it contains an $n$-dimensional torus as
a Zariski open subset in a way such that the
action of the torus on itself extends to
an action of the torus on $X$.
Torus varieties can be completely described
in a purely combinatorial way, via
the concept of toric fans. Information
about this topic can all be found in \cite{CLS}.
We will very briefly go through the concept.
Let $N$ be a lattice of rank $n$
and let $M$ be its dual.
A fan $\Sigma$ in $N\otimes_\mathbb{Z}\mathbb{Q}$
is a collection of cones that is closed under
intersection and cone faces.
Once a toric fan $\Sigma$ is given,
we can define
a toric variety $X_\Sigma$
by gluing the affine toric varieties
${\rm Spec}\,\kk[\sigma^\vee\cap M]$
as $\sigma$ runs through $\Sigma$.

Let $X$ be an irreducible normal variety
with finitely generated divisor class group
${\rm Cl}(X)$ and $\Gamma(X,\mathcal O^*)
\simeq \kk$, i.e. the only global invertible
regular functions are the constants.
The {\em Cox ring} of $X$ is~\cite{ADHL}:
$$\mathcal{R}(X):=\bigoplus_{[D]\in{\rm Cl}(X)}
\Gamma(X,\mathcal{O}_X(D)).$$

In the case that $X$ is also a toric variety,
let $D_1,\ldots,D_r$ be its prime invariant divisors.
It can be shown (cf. \cite[Ch. II, \S 1.3]{ADHL}) that
$$\mathcal{R}(X)=\kk[T_1,\ldots, T_r],
\quad {\rm deg}(T_i)=[D_i].$$
Elements in this ring are said to be
\textit{in Cox coordinates}.

\subsection{The tangent sheaf of a
toric variety}\label{tangenttoric}
Let $X$ be a smooth complete toric variety
with defining fan $\Sigma\subseteq N_\qq$
and character group $M$. 
By the Euler exact sequence for the tangent sheaf
$T_X$ of $X$, the cohomology group of $T_X$ are
graded by $M$. In particular
\[
 H^1(X,T_X) = \bigoplus_{m\in M}H^1(X,T_X)_m.
\]
\begin{defn} (cf. \cite[\S 2.1]{Ilt})
\label{triples}
Let $m\in M$ be such that there exists $\varrho\in\Sigma(1)$
with $m(\varrho) = -1$, where with abuse of notation
we identify the one dimensional cone $\varrho$ with
its primitive generator.
Define the graph $\Gamma_\varrho(m)$ whose 
set of vertices is
\[
 {\rm Vertices}(\Gamma_\varrho(m))
 :=
 \{\varrho'\in\Sigma(1)\setminus\{\varrho\}\, :\,  m(\tau)<0\},
\]
and two vertices are joined by an edge if and only if they are
rays of a common cone of $\Sigma$.
If $C$ is a proper component of $\Gamma_\varrho(m)$ we
say that the triple $(m,\varrho,C)$ is {\em admissible}.
\end{defn}

To any admissible triple $(m,\varrho,C)$ one can associate
a cocycle of $H^1(X,T_X)_m$ in the following 
way. Define a derivation $\partial_{m,\varrho}\in {\rm Der}(\kk[M],
\kk[M])$ by
\[
 \partial_{m,\varrho}\colon \kk[M]\to\kk[M]
 \qquad
 \chi^u\mapsto u(\varrho)\chi^{u+m}.
\]
The announced cocycle is
\begin{equation}
\label{co}
 \xi(m,\varrho,C)
 =
 \{\alpha(\sigma,\tau)\cdot\partial_{m,\varrho}\, :\, \sigma,\tau
 \in\Sigma(n)\}\in H^1(X,T_X)_m
\end{equation}
where $\alpha(\sigma,\tau)$ equals $1$ if 
$\sigma(1)\cap C$ is non-empty
and $\tau(1)\cap C$ is empty, it equals $-1$
if the roles of $\sigma$ and $\tau$ are exchanged
and it equals $0$ otherwise. 
\begin{prop}\cite[Thm 6.5]{IV}
\label{basis-tangent}
The cocycles $\xi(m,\varrho,C)$ span the 
vector space $H^1(X,T_X)$.
\end{prop}

\section{Deformations of smooth toric varieties}
\label{sec:2}

In what follows, $X$ is a smooth toric variety.
Our aim now is to show how to associate to any admissible triple
a one parameter deformation $\pi\colon\mathcal X\to\mathbb A^1$
such that the image of the Kodaira-Spencer map 
$T_{\mathbb A^1}\to H^1(X,T_X)$ associated
to $\pi$ is the original admissible triple.

\subsection{The deformation space}
\label{P-matrix}
Let $X$ be a smooth toric variety and
let $(m,\varrho,C)$ be an admissible triple.
The element $m\in{\rm Hom}(N,\zz)$ is
a homomorphism $N\to\zz$ with kernel $K$.
We let $\gamma\colon\zz\to N$ be the section of 
$m$ defined by $\gamma(-1)=v_\varrho$ and let $\pi\colon
N\to K$ be the corresponding projection.
The above maps are encoded in the following
exact sequence:
\begin{equation}
\label{sequence}
 \xymatrix{
  0\ar[r] & K\ar[r] & N\ar[r]^-m\ar@/^5pt/[l]^-{\pi} 
  &\zz\ar[r]\ar@/^5pt/[l]^-{\gamma} & 0.
 }
\end{equation}

 Let $\tilde N=\zz^2\oplus K\oplus \zz$
 and $\tilde M$ its dual. We define the map
\begin{equation}
\label{emb}
 \imath\colon N\to\ \tilde N
 \qquad
 v\mapsto [m(v),m(v),\pi(v),0].
\end{equation}

Let $\varrho_1,\dots,\varrho_r$ be the primitive generators 
of the one-dimensional cones of the fan $\Sigma$,
let $a_i = m(\varrho_i)$ for any $i$ and let 
\begin{align*}
 U_1 & = \{(1,i) \, : \, a_i > 0\} &
 U_2 & = \{(2,i) \, : \, a_i < 0\text{ and } \varrho_i\in C\cup\{\varrho\}\} \\
 U_4 & = \{(4,i) \, : \, a_i = 0\} &
 U_3 & = \{(3,i) \, : \, a_i < 0\text{ and } \varrho_i\notin C\}
\end{align*}
and let $U$ be the union $U_1\cup U_2\cup U_3\cup U_4$.
For any $(j,i)\in U$ we define the row vector
$v_j = [a_i \, :\, (j,i)\in U_j]$.
Define the matrix $A_j=[\pi(\varrho_i) \, :\, (j,i)\in U_j]$
whose columns are the vectors $\pi(\varrho_j)$.
Finally we define the following block matrix
\begin{equation}
\label{Pmat}
 P(m,\varrho,C)
 :=
 \left[\begin{array}{ccccc}
 1 & v_1 & v_2 & 0 & 0\\
 1 & v_1 &  0 & v_3 & 0\\
  0 & A_1 & A_2 & A_3 & A_4\\
  1 & 0 & 0 & 0 & 0
 \end{array} \right],
\end{equation}
where a $0$ represents a zero matrix of
adequate dimensions, whereas a $1$ is simply
the number one ($1\times 1$ matrix).

From here on, given a one dimensional
ray $\varrho_s$ we denote by $i(\varrho_s)$ 
its index $s$ and by $k(\varrho_s)\in\{1,2,3,4\}$
the index of the set $U_k$ corresponding
to the sign of $a_s$.
Given a maximal cone $\sigma\in\Sigma_X$ 
we define the cone indices of $\tilde\sigma$,
such that the cones $\{\tilde\sigma\}_{\sigma\in\Sigma}$
define the ambient toric variety $\tilde X$ 
where $\mathcal X$ is embedded.
For every
$\varrho_i\in\sigma(1)$ 
we add every possible $(k,i)\in U$ as a cone index for
$\tilde\sigma$. We also add,
if it is not added already,
the index $(2,i(\varrho))$
if $\sigma(1)\cap C=\emptyset$
or the index $(3,i(\varrho))$ if
$\sigma(1)\cap C\neq\emptyset$.
Lastly, we always add 1 as an index
for $\tilde\sigma$.
All this can be summarized as follows:

\begin{itemize}
\item
if $\sigma(1)\cap C$ is empty then the cone indices
for $\tilde\sigma$ are:
\begin{align*}
\{1,(2,i(\varrho))\}
\cup \{(1,s)\, : \, \varrho_s\in\sigma(1)\text{ y }a_s>0\}\\
\cup \{(4,s)\, : \, \varrho_s\in\sigma(1)\text{ y }a_s=0\}
\cup \{(3,s)\, : \, \varrho_s\in\sigma(1)\text{ y }a_s<0\}
\end{align*}
\item
if $\sigma(1)\cap C$ is non-empty then the cone indices
for $\tilde\sigma$ are:
\begin{align*}
\{1,(3,i(\varrho))\}
\cup \{(1,s)\, : \, \varrho_s\in\sigma(1)\text{ y }a_s>0\}\\
\cup \{(4,s)\, : \, \varrho_s\in\sigma(1)\text{ y }a_s=0\}
\cup \{(2,s)\, : \, \varrho_s\in\sigma(1)\text{ y }a_s<0\}.
\end{align*}
\end{itemize}

Denote by $\tilde X$ the toric variety
whose fan $\Sigma_{\tilde X}$,
defined on $\tilde N$,
is given by the cones $\tilde\sigma$
for every $\sigma\in\Sigma_X$.
We then define
$\mathcal X = \mathcal X(m,\varrho,C)$
as the  $T$-variety of
complexity one embedded in $\tilde X$
whose equation in Cox coordinates
$T_1$, $T_{ij}$ with $(i,j)\in U$,
is the following trinomial
\begin{equation}\label{trinomial}
 T_1\prod_{(1,j)\in U_1}T_{1j}^{a_{j}}
 -
 \prod_{(2,j)\in U_2}T_{2j}^{-a_{j}}
 +
 \prod_{(3,j)\in U_3}T_{3j}^{-a_{j}}.
\end{equation}
We denote by $\bar{\mathcal X}$ the 
affine subvariety defined by the above 
trinomial equation in Cox coordinates.

\begin{thm}
\label{one-param}
Let $(m,\varrho,C)$ be an admissible triple and let 
$\mathcal X = \mathcal X(m,\varrho,C)$ be the 
$T$-variety of Construction~\ref{P-matrix}. 
The inclusion $\mathbb K[T_1]\to\mathbb K[\bar{\mathcal X}]$
defines a $T$-equivariant morphism 
$\pi\colon \mathcal X\to\mathbb A^1$
which is a one-parameter deformation of $X$,
i.e. $X$ is isomorphic to the fiber of $\pi$ over 
$0\in\mathbb A^1$. 
\end{thm}

\begin{proof}
To see that $\pi$ is indeed a morphism, recall that
$\mathcal{X}$ is embedded in a toric variety
$\tilde X$ whose
toric fan $\Sigma_{\tilde X}$ has the columns of
$P$ as ray generators and maximal cones given by
$\{\tilde\sigma\}_{\sigma\in\Sigma(n)}$.
By taking the projection
onto the last coordinate, we map every ray of
$\Sigma_{\tilde X}$ to 0, except for the one
corresponding to $T_1$, which is mapped onto $\Q_{\ge 0}$. Thus, we
have a morphism of toric varieties 
\[
 \xx\to\mathbb{A}^1.
\]
Let $\xx_0$ be the fiber of $\pi$
resulting by setting $T_1=0$.
The trinomial in $\eqref{trinomial}$
becomes a binomial
$\chi^{v_1}-\chi^{v_2}$, with $v_1,v_2\in\tilde N$.
Let $v=v_1-v_2$ and let $u\in\tilde M$
be such that $P^*(u)=v$.
Now, 
$\xx_0$ admits an action of the subtorus
defined by $u^\bot$.
Recall that $T_1=0$, so for the action
to be effective we must
take the subtorus $N_0:=u^\bot\cap (e^*_{n+2})^\bot$.
Since $N_0$  has the same dimension as $\xx_0$,
this fiber can be seen as a toric variety having
$$\Sigma_{\xx_0}:=\Sigma_{\tilde X}
\cap N_0$$
as fan.
It can be shown that $\imath(N)=N_0$: Indeed,
a vector $[a,b]\oplus w \oplus [d]\in\tilde N$
belongs to $N_0$ if and only if
$a=b$ and $d=0$, so it is clear that
$\imath(N)\subseteq N_0$. Conversely,
if the vetor is of the form 
$[a,a]\oplus w \oplus [0]\in\tilde N$,
then it is equal to $\imath\left(w
+\gamma(a)\right)$.
We now wish to prove that the following
equality holds
\[
 \Sigma_X=\Sigma_{\xx_0}.
\]
Take a cone $\sigma\in\Sigma_X$
and a ray $\tau\in\sigma(1)$
and let $v_\tau$ be its primitive generator.
If $\tau\in U_1$ or $\tau\in U_4$,
then $\imath(v)\in\tilde\sigma$
because it is a column of $P(m,\varrho,C)$.
Otherwise, $\imath(v)$ is
a linear combination of columns
of $P(m,\varrho,C)$, one
of index $(j_1,i(\tau))$
and one of index $(j_2,i(\varrho))$
with $\{j_1,j_2\}=\{2,3\}$,
thus we still
have $\imath(v)\in\tilde\sigma$
 in this case. We conclude
 that $\imath(v)\in\tilde\sigma\cap\imath(N)=
 \tilde\sigma\cap N_0$.
 Due to the completeness of the
 fans, the fact that $\imath(\sigma)
 \subset \tilde\sigma\cap N_0$ implies
 that $\Sigma_X=\Sigma_{\xx_0}$
as claimed.
\end{proof}

\subsection{The central fiber}\label{centralfiber}
We now describe the embedding $X\to\mathcal X$
at the level of Cox rings. We define the following 
homomorphism of polynomial rings

{\small
\[
 \eta\colon
 \mathbb K[T_{ij} \, :\, (i,j)\in U]
 \to
 \mathbb K[S_1,\dots,S_r]
 \qquad
 T_{ij}\mapsto
 \begin{cases}
 \prod_{(3,j)\in U_2}S_j^{-a_j} & \text{if $i=2$ and $\varrho_j=\varrho$}\\
 \prod_{(2,j)\in U_3}S_j^{-a_j} & \text{if $i=3$ and $\varrho_j=\varrho$}\\
 S_j & \text{otherwise}
 \end{cases}
\]
}
And $T_1\mapsto 0$.
Observe that the variable $T_1$ is the variable which gives the coordinate on the base
$\mathbb A^1$ of the deformation.

\begin{prop}\label{fiberthm}
The homomorphism of polynomial rings $\eta$
induces an isomorphism $\eta'\colon
\mathbb K[\bar{\mathcal X}]/\langle T_1\rangle
\to \rr(X)$ which induces the inclusion $X\to\mathcal X$
in Cox coordinates.
\end{prop}
\begin{proof}
First of all we observe that the binomal 
$\prod_{(2,j)\in S_2}T_{2j}^{-a_{j}}-
\prod_{(3,j)\in S_3}T_{3j}^{-a_{j}}$ is contained
in the kernel of $\eta$. Moreover since the kernel 
is a prime principal ideal we conclude that it is
generated by the above binomial. Thus, after 
identifying $\mathbb K[T_{ij} \, :\, (i,j)\in U]$ with
$\mathbb K[\bar{\mathcal X}]/\langle T_1\rangle$, the homomorphism
$\eta$ induces an isomorphism 
$\eta'\colon\mathbb K[\bar{\mathcal X}]/\langle T_1\rangle\to \mathcal{R}(X)$
as claimed. Observe that $\eta'$ is a graded 
map with respect to the $\Cl(\mathcal X)$-grading
on the domain and the $\Cl(X)$-grading on the 
codomain. 
Denote by 
\[
 \tilde{P}
 :=
 \left[\begin{array}{cccc}
 v_1 & v_2 & 0 & 0\\
 v_1 &  0 & v_3 & 0\\
 A_1 & A_2 & A_3 & A_4
 \end{array} \right].
\]
the matrix obtained by removing the first column 
and the first row from $P(m,\varrho,C)$.
Define the homomorphism 
\begin{equation}\label{matrixnu}
 \psi\colon
 \zz^r\to\zz^{r+1}
 \qquad
 e_j\mapsto
 \begin{cases}
  e_{(1,j)}  & \text{if $a_j>0$}\\
  e_{(2,j)} - a_j e_{(3,\varrho)}
  & \text{if $\varrho_j\in C\cup\{\varrho\}$}\\
  e_{(3,j)} - a_j e_{(2,\varrho)}
  & \text{if $\varrho_j\in (\Gamma_\varrho(m)\setminus C)\cup\{\varrho\}$}\\
  e_{(4,j)}  & \text{if $a_j=0$}
 \end{cases}.
\end{equation}
Observe that $\psi$ fits in the following 
commutative diagram
\begin{equation}\label{comdiag}
 \xymatrix@C=5cm{
  \zz^r\ar[r]^{-\psi}\ar[d]^-{P_X} & \zz^{r+1}\ar[d]^-{\tilde P}\\
  \zz^n\ar[r]^-{(a_1,\dots,a_n)\mapsto (-a_n,-a_n,a_1,\dots, a_{n-1})} & \zz^{n+1}
 }
\end{equation}
where $P_X$ is the $P$-matrix of the Cox construction
of $X$. Moreover $\psi$ maps the positive orthant 
of $\zz^r$ into the positive orthant of $\zz^{r+1}$,
it maps cones of $X$ into cones of $\mathcal X$
and it induces $\eta'$. The statement follows.
\end{proof}

\subsection{The Kodaira Spencer map}

\begin{thm}
\label{ks}
Let $(m,\varrho,C)$ be an admissible triple and let 
$\pi\colon \mathcal X \to\mathbb A^1$ be the 
corresponding one-parameter family.
The image of $\pi$ via the 
Kodaira-Spencer map is the cocycle 
$\xi(m,\varrho,C)\in H^1(X,T_X)_m$
defined in~\eqref{co}.
\end{thm}
\begin{proof}
The complexity one variety $\mathcal X$ is
canonically embedded into the toric variety 
$\tilde X$ and the morphism $\pi\colon\mathcal X\to\mathbb A^1$ 
is induced by a toric morphism $\tilde X\to\mathbb A^1$.
Let $\Sigma$ be the fan of the toric variety $X$.
We denote by $\tilde\Sigma\subseteq\tilde N_\qq$ 
the fan of $\tilde Z$.
Given a cone $\sigma\in\Sigma$ we denote by 
$\tilde\sigma$ the corresponding cone of $\tilde\Sigma$,
that is $\tilde\sigma\cap\imath(N) = \sigma$,
where the map $\imath$ is the one defined in~\eqref{emb}.
Let $\tilde M$ be the dual of $\tilde N$.
The trinomial~\eqref{trinomial} is locally
described in $\mathbb K[\tilde\sigma^\vee\cap\tilde M]$
by a polynomial of the form
\[
 \chi^{u_4+u_1}-\chi^{u_2}+\chi^{u_3},
\]
where $u_4 = [0,\dots,0,1]$, so that 
$\varepsilon = \chi^{u_4}$. We denote by
$\varrho_\sigma$ the primitive generator 
of the extremal ray of the cone $\tilde\sigma$ 
which is one of the column of the $P$-matrix~\eqref{Pmat}
whose index is $(2,i(\varrho))$ if 
$\sigma\cap C$ is non-empty and 
it is $(3,i(\varrho))$ otherwise.
Assume we are in the first case then
the following equation holds
\[
 P^*(u_2) = v_2,
\]
where $T^{v_2}$ is the monomial in Cox coordinates
which corresponds to the character $\chi^{u_2}$.
The monomial $T^{v_2}$ does not contain any
variable $T_{(k,i)}$ such that $\varrho_i\in\sigma(1)$
with the only exception of the variable $T_{(2,i(\varrho))}$
which appears with exponent $1$. This implies
that $u_2$ has scalar product $0$ with each
column of the $P$-matrix of index 
$(k,i)$ when $\varrho_i\in\sigma(1)\setminus\{\varrho\}$
and it has scalar product $1$ with the column
of index $(2,i(\varrho))$. In particular $u_2$ generates
an extremal ray of the smooth cone $\tilde\sigma$ and then
$\chi^{u_2}$ is a variable of the polynomial ring
$\mathbb K[\tilde\sigma^\vee\cap\tilde M]$.
Analogously, if $\sigma\cap C$ is empty,
the character $\chi^{u_3}$ is a variable
of the ring.
Both cases establish
an isomorphism
\[
 \dfrac{\kk[\tilde\sigma^\vee\cap \tilde M]}
 {\langle\chi^{u_4+u_1}-\chi^{u_2}+\chi^{u_3}\rangle}
 \to
 \kk[\tilde\sigma^\vee\cap
 \tilde M\cap\varrho_\sigma^\bot].
\]
For the rest of this proof, we fix two cones $\sigma,\tau
\in\Sigma$.
The isomorphism above
leads to the following diagram

\[
 \xymatrix{
  \dfrac{\kk[(\tilde\sigma\cap\tilde\tau)^\vee\cap 
  \tilde M]}
  {\langle\chi^{2u_4},\chi^{u_4+u_1}-\chi^{u_2}+\chi^{u_3}\rangle}\ar[d]^-{\simeq}
  \ar@{=}[r] &
  \dfrac{\kk[(\tilde\sigma\cap\tilde\tau)^\vee\cap 
  \tilde M]}
  {\langle\chi^{2u_4},\chi^{u_4+u_1'}-\chi^{u_2'}+\chi^{u_3'}\rangle}\ar[d]^-{\simeq}\\
  \dfrac{\kk[(\tilde\sigma\cap\tilde\tau)^\vee\cap \tilde M\cap\varrho_\sigma^\bot]}
  {\langle\chi^{2u_4}\rangle}\ar[d]^-{\simeq}_\beta
  &
  \dfrac{\kk[(\tilde\sigma\cap\tilde\tau)^\vee\cap \tilde M\cap\varrho_\tau^\bot]}
  {\langle\chi^{2u_4}\rangle}\ar[d]^-{\simeq}\\
  \kk[(\sigma\cap\tau)^\vee\cap M]\otimes_\kk\kk[\varepsilon]
 &
  \kk[(\sigma\cap\tau)^\vee\cap M]\otimes_\kk\kk[\varepsilon]
 }
\]
where the map $\beta$ is defined by the composition
$\beta^*\colon M\cap\varrho_\sigma^\bot \to M
\to M/\langle u\rangle$ of the inclusion with the projection
and observing that $\beta^*(\tilde\sigma)=\sigma$ and 
$\beta^*(\tilde\tau)=\tau$.
Moreover $\beta^*$ is an isomorphism being 
$u(\varrho_\sigma) = \pm 1$. 
We have thus constructed an
isomorphism
\[
\varphi\colon
 \kk[(\sigma\cap\tau)^\vee\cap M]\otimes_\kk\kk[\varepsilon]
 \to
 \kk[(\sigma\cap\tau)^\vee\cap M]\otimes_\kk\kk[\varepsilon].
\]

Let $a\in(\tilde\sigma\cap\tilde\tau)^\vee \cap M$.
We will assume that
$\chi^{u_2}$ is a variable in
$\kk[\tilde\sigma^\vee\cap\tilde M]$
and $\chi^{u_3'}$ is a variable in
$\kk[\tilde\tau^\vee\cap\tilde M]$.
The other cases work similarly so analyzing
only this case is enough.
Since $\tau^\vee$ is smooth, we can
write
$a=v + a(\varrho_\tau) u_3'$
where $v$ is a linear combination of
the rays of $\tau^\vee$ different from $u_3'$.
Hence the following hold
\begin{align*}
 \beta(\chi^a)
 & = \beta(\chi^v\chi^{ a(\varrho_\tau) u_3'})\\
 & = \beta(\chi^v(\varepsilon\chi^{u_1'}+\chi^{u_3'})^n)\\
 & = \chi^{\imath^*(a)}+ a(\varrho_\tau)\varepsilon
        \chi^{\imath^*(a)+\imath^*(u_1'-u_3')}
\end{align*}
where the last equality is due to the fact
that the argument of $\beta$ does not contain $\chi^{u_2'}$.
Now, observe that $\chi^{u_1-u_3}$ is the monomial
$T^w$ where $w$ is the difference between the second and
last row of the matrix $P$. This means $u_1-u_3=
[0,1,0,\ldots,0,-1]$ and therefore $\imath^*(u_1-u_3)=m$.
By setting $u=\imath^*(a)$,
this shows that $\varphi$ is defined as
$$\varphi(\chi^u)=
 a(\varrho_\tau) \chi^{u+m}.$$
The only thing left to prove is that
$a(\varrho_\tau)=
u(\varrho)$. Simply notice that
$a(\varrho_\sigma)=0$, so
$$a(\varrho_\tau)=
a(\varrho_\tau+\varrho_\sigma)=
a\left(\imath(\varrho)\right)
= \imath^*(a)\left(\varrho\right)
=u(\varrho).$$
The coefficient $\alpha(\sigma,\tau)$
from \eqref{co} equals 1 in this case and
is easily seen to appear when checking the
other cases.
\end{proof}

\section{Polyhedral description}
\label{sec:3}

In this section we describe
deformations of toric varieties
as shown in \cite{IV}. Their
results are very closely related to the ones found
in Section \ref{sec:2} of this paper,
but they use a completely different
language. Namely, the language
of polyhedral divisors, which we
summarize here.

\subsection{T-varieties}
Let $X$ be an algebraic variety over $\kk$ having an action of
$T:=(\kk^*)^n$ (which is called the $n$-dimensional torus).
This action is called \textit{effective} if the only $t\in T$,
for which $t\cdot x=x$ holds for all $x\in X$, is the identity of $T$.
A \textit{$T$-variety} is a normal algebraic variety $X$
coming with an effective $(\kk^*)^n$-action.
The \textit{complexity} of $X$ is the difference $\dim X-n$.
These varieties admit a polyhedral description
given by  K. Altmann and J. Hausen for the affine case in
\cite{AH} and later, together with H. Su\ss,
in \cite{AHS} for the non-affine case.
In the following section, we briefly recall this construction.

\subsection{Polyhedral divisors}
Let $N$ be a lattice of rank $n$ and $M=\textrm{Hom}\langle N,\zz\rangle$ its dual. 
We denote by 
$N_\Q:=N\otimes_\zz\Q$ and by $M_\Q:=M\otimes_\zz\Q$
the rational vector spaces.
A \textit{polyhedron} in $N_\Q$ is an intersection of finitely many affine half spaces in $N_\Q$. If we require the supporting hyperplane of any half space to be a linear subspace, the polyhedron is called a \textit{cone}. If $\sigma$ is a cone in $N_\Q$, its \textit{dual cone} is defined as
$$\sigma^\vee:=\{u\in M_\Q: u(v)\ge 0\textrm{ for all }v\in N_\Q\}.$$
Let $\Delta\subseteq N_\Q$ be a polyhedron. The set
$$\sigma:=\{v\in N_\Q:tv+\Delta\subseteq\Delta,\ \forall t\in\Q\}$$
is a cone called the \textit{tailcone} of $\Delta$ and $\Delta$ is called a \textit{$\sigma$-polyhedron}.
Let $Y$ be a normal variety and $\sigma$ a cone. A \textit{polyhedral divisor} on $Y$
is a formal sum
$$\dd:=\sum_P \Delta_P\otimes P,$$
where $P$ runs over all prime divisors of $Y$ and the $\Delta_P$ are all $\sigma$-polyhedrons such that $\Delta_P=\sigma$ for all but finitely many $P$. We admit the empty set as a valid
$\sigma$-polyhedron too.
 Let $\mathfrak{D}:=\sum\Delta_P\otimes P$ be a polyhedral divisor on $Y$, with tailcone $\sigma$. For every $u\in\sigma^\vee$ we define the evaluation
$$\mathfrak{D}(u):=\sum_{\substack{P\subset Y \\ \Delta_P\neq\emptyset}}\min_{v\in\Delta_P} u(v)\otimes P\in{\rm WDiv}_{\Q}({\rm Loc}\,\dd)$$
where ${\rm Loc}\,\dd:=Y\setminus (\cup_{\Delta_P=\emptyset}P)$ is the \textit{locus} of $\dd$.

\begin{defn}\label{pp}
Let $Y$ be a normal variety. A \textit{proper polyhedral divisor}, also called a \textit{pp-divisor} is a polyhedral divisor $\mathfrak{D}$ on $Y$, such that
\begin{enumerate}[(i)]
\item $\mathfrak{D}(u)$ is Cartier and semiample for every $u\in\sigma^\vee\cap M$.
\item $\mathfrak{D}(u)$ is big for every $u\in(\textrm{relint}\,\sigma^\vee)\cap M$.
\end{enumerate}
\end{defn}

Now, let $\mathfrak{D}$ be a pp-divisor on a semiprojective (i.e. projective over some affine variety) variety $Y$, $\mathfrak{D}$ having tailcone $\sigma\subseteq N_\Q$. This defines an $M$-graded algebra
$$A(\mathfrak{D}):=\bigoplus_{u\in\sigma^\vee\cap M}\Gamma({\rm Loc}\,\dd,\oo(\dd(u))).$$
The affine scheme $X(\dd):={\rm Spec}\,A(\dd)$ comes with a natural action of ${\rm Spec}\,\kk[M]$.
Definition~\ref{pp} is mainly motivated
by the following result~\cite{AH}*{Theorem 3.1 and Theorem 3.4}.

\begin{thm}
Let $\mathcal D$ be a pp-divisor on a 
normal variety $Y$.
Then $X(\dd)$ is an affine T-variety of complexity equal to $\dim Y$.
Moreover, every affine $T$-variety arises like this.
\end{thm}
\subsection{Divisorial fans}
Non-affine $T$-varieties are obtained by gluing 
affine $T$-varieties coming from pp-divisors in
a combinatorial way as specified in Definition~\ref{glue}.
Consider two polyhedral divisors $\mathcal{D}=\sum\Delta_P\otimes P$ and $\mathcal{D'}=\sum\Delta_P'\otimes P$ on $Y$, with tailcones $\sigma$ and $\sigma'$
respectively and such that $\Delta_P\subseteq\Delta_P'$
for every $P$. We then have an inclusion
$$\bigoplus_{u\in\sigma^\vee\cap M}\Gamma({\rm Loc}\,\dd,\oo(\dd(u)))\subseteq \bigoplus_{u\in\sigma^\vee\cap M}\Gamma({\rm Loc}\,\dd,\oo(\dd'(u))),$$
which induces a morphism $X(\dd')\rightarrow X(\dd)$. We say
that $\dd'$ is a \textit{face} of $\dd$, denoted by $\dd'\prec\dd$, if this morphism is an open embedding.

\begin{defn}\label{glue}
A \textit{divisorial fan} on $Y$ is a finite set $\mathcal{S}$ of pp-divisors on $Y$
such that for every pair of divisors $\dd=\sum\Delta_P\otimes P$ and $\dd'=\sum\Delta_P'\otimes P$
in $\mathcal{S}$, we have $\dd\cap\dd'\in\mathcal{S}$ and
$\dd\succ\dd\cap\dd'\prec\dd'$, where 
$\dd\cap\dd':=\sum(\Delta_P\cap\Delta_P')\otimes P$.
\end{defn}

This definition allows us to glue affine $T$-varieties via
$$X(\dd)\longleftarrow X(\dd\cap\dd')\longrightarrow X(\dd'),$$
thus resulting in a scheme $X(\mathcal{S})$.
For the following theorem see~\cite{AHS}*{Theorem 5.3 and Theorem 5.6}.

\begin{thm}
The scheme $X(\mathcal{S})$ constructed above is a $T$-variety of complexity equal to $\dim Y$.  Every $T$-variety
can be constructed like this.
\end{thm}

\subsection{Deformations via polyhedral
decompositions}
In \cite[\S 6]{IV}, Ilten and Vollmert define one-parameter
deformations of smooth toric varieties
denoted by
$\pi=\pi(m,\varrho,C)$, where $m$ is a lattice
vector, $\varrho$ is a ray in a fan and $C$ is a connected
component of some graph.
The construction is as follows.

Let $\Sigma$ be a smooth complete fan giving rise to
a toric variety $X=X_\Sigma$. By choosing a $m\in M$
and intersecting the hyperplanes $\{v\in N :  m(v)=-1\}$
and $\{ v\in N :  m(v)=1\}$ with
$\Sigma$ we get two polyheral subdivisions
corresponding to the slices $\Ss_0$ and $\Ss_\infty$ of a
divisorial fan $\Ss$ on $\pp^1$,
describing $X$ as a variety of complexity one. Now, choose
$\varrho\in\Sigma(1)$ such that $ m(\varrho)=-1$ and
recall Definition~\ref{triples} of the graph $\Gamma_\varrho(m)$, 
whose set of vertices is
$$\{\tau\in\Sigma(1):\tau\neq\varrho,  m(\tau)<0\}$$
and whose edges join two vertices whose corresponding
one-dimensional rays lie in a common cone.
Assume $\Gamma_\varrho(m)$ has at least two connected
components, and let $C$ be one of them. This choice
induces a one parameter deformation on $X$ as follows.

Each polyhedron $\Delta\in\Ss_0$ will be decomposed as
$\Delta=\Delta^0+\Delta^1$. If $\Delta$ contains a vertex coming
from a ray in $C$, take $\Delta^0={\rm tail}\,\Delta$ and
$\Delta^1=\Delta$. If $\Delta$ contains no such vertex,
take $\Delta^0=\Delta$ and $\Delta^1={\rm tail}\,\Delta$.
The sets $\{\Delta^0\}_{\Delta\in\Ss_0}$ and $\{\Delta^1\}_{\Delta\in\Ss_0}$
define new  polyhedral subdivisions $\Ss_0^0$ and $\Ss_0^1$
such that $\Ss=\Ss_0^0+\Ss_0^1$. Let $\tilde\Ss$
be the divisorial fan on $\mathbb{A}^1\times\pp^1$ whose only
non-trivial slices are $\Ss_0^0$ at $V(y)$, $\Ss_0^1$ at $V(y-x)$
and $\Ss_\infty$ at $V(y^{-1})$, where we are using coordinates
$(x,y)\in\mathbb{A}^1\times\pp^1$. Then $\mathcal{X}:=X(\tilde\Ss)$
comes with a morphism $\pi:\mathcal{X}\rightarrow\mathbb{A}^1$ which
is a one-parameter deformation of $X$. This deformation is called $\pi(m,\varrho,C)$. 

\begin{prop}
The deformation $\pi(m,\varrho,C)$
is the same as the one described in Theorem \ref{one-param}.
\end{prop}

\begin{proof}
We consider the $\kk^*$-action on
$\mathbb{A}^1\times\pp^1$ given by
$t\cdot(x,y)=(tx,ty)$. This allows us to
describe this surface with a divisorial fan $\mathcal{Z}$
on $\pp^1$ whose tailfan is given by a single
ray on the positive axis and whose only
non-trivial slice has vertices in 0 and 1.
By applying \cite[Prop 2.1]{IV2},
we describe $\xx$ with a new divisorial fan $\Ss'$
on $\pp^1$, having three non-trivial slices.
One non-trivial slice of $\Ss'$ contains $\Ss_0^0$ at
height 0 and a single vertex at height 1, whereas
the other two non-trivial slice of $\Ss'$ are
simply $\Ss_0^1$ and $\Ss_\infty$ embedded
in the corresponding higher dimensional space.
Thus, by \cite[Corollary 4.9]{HS} we
see that the cox ring of $\xx$ is given precisely
by \eqref{trinomial}.
The matrix $P(m,\varrho,C)$ can be
obtained from \cite[Pop 4.7]{HS}.
\end{proof}

\section{Applications}
We use the language developed in
subsection \ref{tangenttoric} and
in section \ref{sec:2} to study
deformations of scrolls and 
deformations of hypersurfaces
of smooth toric varieties.

\subsection{Deformations of scrolls}
Let $n>1$ be an integer and let
$a_1,\ldots,a_n$ be integers. 
We denote by $\mathbb{F}(a_1,\ldots,a_n)$
the $\pp^{n-1}$-bundle (i.e. a scroll)
over $\pp^1$ associated to the sheaf
$\oo_{\pp^1}(a_1)+\ldots+
\oo_{\pp^1}(a_n)$. It can be
defined as the quotient of the space
$(\A^2\backslash 0)\times(\A^n\backslash 0)$
by the following $(\kk^*)^2$-action.

\begin{eqnarray*}(\lambda,1)\cdot (t_1,t_2,x_1,\ldots,x_n)
&=&(\lambda t_1, \lambda t_2,
\lambda^{-a_1} x_1,\ldots,
\lambda^{-a_n}x_n)\\
(1,\mu)\cdot (t_1,t_2,x_1,\ldots x_n)
&=& (t_1,t_2,\mu x_1,\ldots,
\mu x_n).\end{eqnarray*}
The action on the
fist two coordinates gives
$\F(a_1,\ldots,a_n)$ a morphism
over $\pp^1$ by projecting
on the first factor

\[
 \xymatrix@C=3cm{
  (\A^2\backslash 0)\times(\A^n\backslash 0)\ar[r]\ar[d] & \F(a_1,\ldots,a_n)\ar[d]\\
   (\A^2\backslash 0)\ar[r]&\pp^1
 }
\]

\begin{rem}\label{isomscroll}
It can be shown (cf. \cite[Ch. 2]{Rei}) that
$\F(a_1,\ldots,a_n)\cong\F(b_1,\ldots,b_n)$
if and only if there exists $c\in\zz$
and a permutation $\sigma\in S_n$ such
that for every $i$ we have
$a_i=b_{\sigma(i)}+c$.
\end{rem}


\begin{prop}\label{deformscroll}\ 
\begin{enumerate}[(a)]
\item A scroll over $\pp^1$
is rigid if
and only if it is isomorphic to
$\mathbb{F}(a_1,\ldots,a_n)$
where $\{a_i\}_{i=1}^n\subseteq\{0,1\}$.

\item Let $X=\mathbb{F}(a_1,\ldots,a_n)$,
such that $d:=a_1-a_2>2$. For any
$d'<d$, the scroll $X$ admits a
deformation to $\mathbb{F}(a_1-d',a_2+d',a_3,
\ldots,a_n)$.
\end{enumerate}
\end{prop}

\begin{proof}
If $n=2$, we have a Hirzebruch surface and
these results are well known.
They can be found for example in
\cite[\S 3]{Ilt}. Therefore, we will assume $n\ge 3$.
The degree matrix of $\F(a_1,\ldots, a_n)$
is
$$Q=\left[
\begin{array}{rrrcr}
1&1&-a_1&\cdots&-a_n\\
0&0&1&\cdots&1
\end{array}
\right].$$
Since the components of
the irrelevant ideal are
$(1,2)\ \textrm{and}\ (3,\ldots,n+2)$
and $n\ge 3$,
it is clear that $(\varrho_1,\varrho_2)$ is
the only pair
of rays in $\Sigma$ that are not in a common
cone. Thus, if we choose
an admissible triple $(m,\varrho,C)$
we must have that
\begin{itemize}
\item $m(\varrho_1)<0$ and
$m(\varrho_2)<0$. 
\item $ m(\varrho_k)=-1$
for some $3\le k\le n+2$. This $\varrho_k$
will be $\varrho$.
\item $m(\varrho_i)\ge 0$
for every $i=3,\ldots,n+2$, $i\neq k$. 
\end{itemize}
since this conditions are the only
way to ensure that $\Gamma_\varrho(m)$
has at least two connected components.
Now, define $u_i:=m(\varrho_i)$
and form the column vector
$$u:=(u_1,\ldots,u_{n+2})^t.$$
The conditions above become
\begin{equation}\label{usigns}
u_1,u_2<0;\quad \ u_k=-1;\quad \ u_i\ge 0,\ i\notin\{1,2,k\}.\end{equation}
From the Gale duality
(cf. \cite[Ch. II, \S 1.2]{ADHL}) we see
that $Q(u)=0$, which when written
as a system of equations is
equivalent to

\begin{numcases}{}
u_1+u_2-\sum_{i=3}^{n+2}u_ia_{i-2}=0\label{usysteq}\\
\sum_{i=3}^{n+2}u_i=0\label{usysteq2}
\end{numcases}
From $\eqref{usigns}$ and $\eqref{usysteq2}$
we deduce there exists $3\le \ell\le n+2$,
with $\ell\neq k$, such that
\begin{equation}\label{1and0}
u_\ell=1\ {\rm and}\ 
u_i=0\ {\rm for}\  i\notin\{1,2,k,\ell\}.\end{equation}

Aditionally, by $\eqref{usysteq}$,
\begin{equation}\label{gttwo}
a_{k-2}-a_{\ell-2}=-u_1-u_2\ge 2.
\end{equation}

Thus, $\mathbb{F}(a_1,\ldots,a_n)$
has an admissible triple if and only
if two of the $a_i$ have distance
at least 2,
proving part (a).

Assume now that we are in the case
where the admissible triple $(m,\varrho,C)$
exists. We will set $C:=\{\varrho_1\}$
From here on, to simplify notation
without loss of generality, let
$k=3$ and $\ell=4$.
Recall that $u_i= m(\varrho_i)$.
Therefore, \eqref{usigns}
and \eqref{1and0} imply that
the trinomial of the Cox ring of
the total deformation space $\xx$ is
\begin{equation}\label{trincox}
T_1T_{(1,4)}-T_{(2,3)}T_{(2,1)}^{-u_1}
+T_{(3,3)}T_{(3,2)}^{-u_2}\end{equation}
The irrelevant ideal $\mathcal{I}$
of $\xx$
is given by the following components
\begin{eqnarray*}
\mathcal{I}_1&=&\left\langle T_{(2,1)}, T_{(3,2)}\right\rangle\\
\mathcal{I}_2&=&\left\langle
T_{(2,1)},T_{(2,3)},T_{(1,4)}\right\rangle + \left\langle T_{(4,5)},T_{(4,6)},
\ldots,T_{(4,n+2)}\right\rangle\\
\mathcal{I}_3&=& \left\langle
T_{(3,2)},T_{(3,3)},T_{(1,4)}\right\rangle + \left\langle T_{(4,5)},T_{(4,6)},
\ldots,T_{(4,n+2)}\right\rangle\\
\mathcal{I}_4&=&\left\langle
T_{(2,3)},T_{(3,3)},T_{(1,4)}\right\rangle +\left\langle T_{(4,5)},T_{(4,6)},
\ldots,T_{(4,n+2)}\right\rangle
\end{eqnarray*}
In the general fiber of the deformation,
 we put $T_1=t\in\kk^*$,
 so by \eqref{trincox}, the variable $T_{(1,4)}$ can
 be replaced by the other variables.
 This reduces $\mathcal{I}$ to
 $$\langle T_{(2,1)},T_{(3,2)}\rangle\ ,\ 
\left\langle
T_{(2,3)},T_{(3,3)}\right\rangle +\left\langle
T_{(4,5)},T_{(4,6)},
\ldots,T_{(4,n+2)}\right\rangle.$$
By Proposition \ref{fiberthm},
and using the same notation, we have
 $${\rm deg}\left(T_{(2,3)}\right)
 ={\deg}\left(S_2^{-u_2}S_3\right);
 \quad {\rm deg}\left(T_{(3,3)}\right)
={\rm deg}\left(S_1^{-u_1}S_3\right).$$
This means that both the irrelevant ideal
and the degree matrix of the general fiber
of the deformation match that of
$\mathbb{F}(a_1+u_1,a_1+u_2,a_3,\ldots,a_n)$.
Then (b) follows after noticing that \eqref{gttwo}
implies $a_1+u_2=a_2-u_1$.
\end{proof}

\begin{prop}
The scroll $\F(a_1,\ldots,a_n)$ can be deformed to
$$\F(\underbrace{1,1,\ldots,1}_r,
\underbrace{0,0\ldots,0}_{n-r})$$ where
$$r\equiv \sum_{i=1}^n a_i\ {\rm (mod\ }n{\rm )}.$$
\end{prop}

\begin{proof}
By Remark \ref{isomscroll}, we
can assume that the sequence
$a_1,\ldots,a_n$ is decreasing and non-negative,
with $a_n=0$.
We proceed by induction over $a_1$.
The cases $a_1=0$ and $a_1=1$ are trivial.
Assume now that $a_1\ge 2$. Let
$$M=\#\{i : a_i=a_1\},\quad
m=\#\{i : a_i=0\}.$$
If $M< m$, then by Proposition \ref{deformscroll}
the scroll can be deformed
by subtracting 1 from each $a_1,\ldots,a_M$
and adding 1 to $M$ of the $a_i$ that equal 0.

If $M\ge m$, the scroll can be deformed by
subtracting 1 from each $a_1,\ldots,a_m$
and adding 1 to every $a_i$ that equals 0.
Then we subtract 1 from every $a_i$ (recall that this
does not change the variety).

Note that in both cases, and after just a permutation
of indices, we have deformed the original scroll
to $\F(b_1,\ldots,b_n)$ where $b_i\ge b_{i+1}$
for every $i$, the $b_i$ are all
non-negative and $b_n=0$. Furthermore
$\sum a_i \equiv \sum b_i$ (mod $n$) and
$b_1<a_1$
so the induction is complete.
\end{proof}

\subsection{Deformation of hypersurfaces}

Let $X$ be a toric variety and let $\kk[S_1,\ldots, S_r]$
be its Cox ring. Choose an admissible
triple $(m,\varrho,C)$ and construct the corresponding
deformation as explained in section \ref{P-matrix}.
The map $\eta$ given in section \ref{centralfiber}
is defined by a semigroup homomorphism
$\nu_+\colon \zz_{\ge 0}^{r+1}
\rightarrow \zz_{\ge 0}^{r}$
which can be naturally extended to
a group homomorphism
$\nu\colon \zz^{r+1}
\rightarrow \zz^{r}$. Notice that $\nu$
is the transpose of $\psi$ defined in \eqref{matrixnu}.
A homogeneous polynomial $f\in\kk[S_1,\ldots,S_r]$ can be written as a sum
$$f=c_1\mathfrak{m}_1 +\ldots + c_k\mathfrak{m}_k$$
where $c_i\in\kk$ and $\mathfrak{m}_i$ is a monomial for all $i$.
A homogeneous polynomial $\tilde{f}\in
\kk[T_1,T_{i,j}]$ such that $f=\eta(\tilde{f})$ will exist
if and only if the exponent vector of each $\mathfrak{m}_i$
is in the image of $\nu_+$. In this case, if we
let $g\in\kk[T_1,T_{ij}]$
be the trinomial \eqref{trinomial} corresponding
to $(m,\varrho,C)$, the subvariety
$$V(\tilde{f},g)\subset \tilde{X}$$
defines a one-parameter deformation of $X$. Observe that
if $\tilde{f}'\in\kk[T_1,T_{ij}]$ is another lifting of $f$,
i.e. $\eta(\tilde{f}')=f$, then $\tilde{f}'-\tilde{f}
\in\langle g,T_1\rangle$ so that the equality
$V(\tilde{f},g,T_1)=V(\tilde{f}',g,T_1)$ holds.

Let $Q_X\colon \zz^r\rightarrow {\rm Cl}(X)$ be the
grading map of the toric variety $X$, i.e $Q_X$ maps
$e\in\zz^r$ to the class of the divisor $\sum_{i=1}^r
e_iD_i$, where $D_i$ is the $i$-th invariant prime divisor of
$X$. Given a class $w\in {\rm Cl}(X)$ and an
equivariant divisor $D$ of $X$ such that $[D]=w$,
a monomial basis of the Riemann-Roch space of
$D$ is in bijection with the set
$$Q_X^{-1}(w)\cap\zz^r_{\ge 0}.$$
The subset of monomials that can be lifted to
monomials of $\kk[T_1,T_{ij}]$ via $\eta$ is in bijection
with
$${\rm im}(\nu_+)\cap Q_X^{-1}(w)\cap\zz^r_{\ge 0}.$$

\begin{prop} The set ${\rm im}(\nu_+)$
is the Hilbert basis of the rational 
polyhedral cone that it generates.
\end{prop}

\begin{proof}
Let $A_\nu$ be the matrix associated to the
map $\nu$ and let $j_\varrho$
be the index such that $S_{j_\varrho}$ corresponds
to the ray $\varrho$ in $\Sigma_X$.
Due to the way $\eta$ is
defined, it is clear that by removing the $(2,j_\varrho)$-th
column from $A_\nu$, and after an adequate rearrangement
of its columns, we obtain a matrix with the following properties:
\begin{itemize}
\item All the entries in the diagonal are 1.
\item Only one column has non-zero entries outside of the diagonal.
\end{itemize}
It is easy to see that such a matrix has determinant equal to 1.

Similarly, we can remove the $(3,j_\varrho)$-th column
from $A_\nu$ to get a matrix with determinant 1. This
shows that the cone generated by the columns
of $A_\nu$ is the union of two smooth cones
(in the sense of toric geometry), which
proves the statement.

\end{proof}

Let $P_X$ be the matrix whose columns
are the generators of the rays of $\Sigma_X$.
Let $\tilde P$ be the the minor of $P(m,\varrho,C)$
resulting from removing the leftmost column and bottom row.
Let $\tilde Q$ be the cokernel of $\tilde P^*$, i.e.
the grading matrix of $\tilde X$ after removing
the null vector column corresponding to $T_1$.
From the Cox construction,
we get the following commutative diagram
of group homomorphisms
with exact rows
\[
 \xymatrix@C=2cm{
  0\ar[r]&\tilde{M}\ar[r]^{\tilde{P}^*}\ar[d]^\xi &
  \zz^{r+1}\ar[r]^{\tilde{Q}} \ar[d]^{\nu}
  & {\rm Cl}(\tilde X)\ar[d]^{\bar\nu}\ar[r]&0\\
  0\ar[r]&M\ar[r]^{P_X^*} & \zz^r\ar[r]^{Q_X}&{\rm Cl}(X)\ar[r]&0
 }
\]
where the square on the left is the dual of
\eqref{comdiag} and
$\bar\nu$ is uniquely defined by $\nu$.
Denote the exponent vector of a monomial $\mathfrak{m}$ by
${\rm\bf v}(\mathfrak m)$. Then we have
$$\textstyle\ker\nu=
{\rm\bf v}\left( \prod_{(2,j)\in U_2}T_{2j}^{-a_{j}}\right)
 - {\rm\bf v}\left(\prod_{(3,j)\in U_3}T_{3j}^{-a_{j}} \right)
 \subseteq \ker\tilde{Q},$$
which together with the surjectivity of $\xi$ and $\nu$,
imply that $\bar\nu$ is an isomorphism.

\begin{ex}
We now turn our attention to the case
of Hirzebruch surfaces, i.e. $X=\mathbb{F}_n$.
The fan $\Sigma_X$ has four rays $\rho_1,
\rho_2,\rho_3,\rho_4$, generated respectively by
$$v_1=(1,0),\ 
v_2=(0,1),\ v_3=(-1,n),\ v_4=(0,-1).$$
Let $D_1,D_2,D_3,D_4$ be the corresponding invariant
divisors.
In this case we have ${\rm Cl}(X)\cong \zz^2$
generated by $[D_1]$ and $[D_2]$, along with
$$P_X=\left[\begin{array}{rrrr}
1&0&-1&1\\ 0&1&n&-1
\end{array}
\right],\quad
Q_X=\left[\begin{array}{rrrr}
1&0&1&n\\ 0&1&0&1
\end{array}\right],$$
plus a section $s$ for $Q$ and
a projection $\pi$ for $P^*$ given by
$$s=\left[\begin{array}{rr}
1&0\\0&1\\0&0\\0&0
\end{array}
\right],\quad
\pi=\left[\begin{array}{rrrr}
0&0&-1&0\\ 0&0&0&-1
\end{array}\right].$$
We choose $\omega=(a,b)\in {\rm Cl}(X)$, corresponding
to the class $a[D_1]+b[D_2]$, with 
$a> bn> 0$ to guarantee ampleness.
Then, by \cite[\S 9.1]{CLS}, the polyhedron
$\pi(Q_X^{-1}(\omega)\cap \zz^r_{\ge 0})$
is a trapezoid with vertices (in counterclockwise order)
$$(0,0),\ (-a,0),\ (-a,-b),\ (-nb,-b).$$
Applying $P^*+(a,b,0,0)$ to it, we
obtain the trapezoid $Q_X^{-1}(\omega)\cap\zz^{r}_{\ge 0}$,
whose vertices are
$$(a,b,0,0),\ (0,b,a,0),\ (0,0,a-bn,b),\ (a-bn,0,0,b).$$
If we now consider the deformation given by
the admissible triple $(m,\varrho,C)$
where $m=[-\alpha,-1],\ \varrho=\rho_2,\ C=\{\rho_1\}$
and $0<\alpha <n$, we get
$$\tilde P=\left[\begin{array}{rrrrc}
1&-\alpha&-1&0&0\\ 1&0&0&-1&{\scriptstyle\alpha -n}\\
0&1&0&0&-1\ \ 
\end{array}
\right],\quad
\nu=\left[\begin{array}{rrcrr}
0&1&0&\alpha&0\\ 0&0&1&1&0 \\ 
0&0&{\scriptstyle n-\alpha}&0&1 \\ 1&0&0&0&0
\end{array}\right].$$
Notice that the vertices of
$Q_X^{-1}(\omega)\cap\zz^{r}_{\ge 0}$ can now
be written as
$$\nu(0,a-b\alpha,0,b,0),\ \nu(0,0,b,0,a-bn+b\alpha),\ 
\nu(b,0,0,0,a-bn),\ \nu(b,a-nb,0,0,0),$$
which shows that the trapezoid is contained in ${\rm im}(\nu_+)$.
This means that when we deform Hirzebruch surfaces,
every function in the Riemann-Roch space
of the class $\omega$ can be lifted via $\eta$.


\end{ex}

$$$$


\begin{bibdiv}
\begin{biblist}

\bib{ADHL}{book}{
   author={Arzhantsev, Ivan},
   author={Derenthal, Ulrich},
   author={Hausen, J{\"u}rgen},
   author={Laface, Antonio},
   title={Cox rings},
   series={Cambridge Studies in Advanced Mathematics},
   volume={144},
   publisher={Cambridge University Press, Cambridge},
   date={2015},
   pages={viii+530},
   isbn={978-1-107-02462-5}
}

\bib{AH}{article}{
   author={Altmann, Klaus},
   author={Hausen, J{\"u}rgen},
   title={Polyhedral divisors and algebraic torus actions},
   journal={Math. Ann.},
   volume={334},
   date={2006},
   number={3},
   pages={557--607},
   issn={0025-5831}
}

\bib{AHS}{article}{
   author={Altmann, Klaus},
   author={Hausen, J{\"u}rgen},
   author={S{\"u}ss, Hendrik},
   title={Gluing affine torus actions via divisorial fans},
   journal={Transform. Groups},
   volume={13},
   date={2008},
   number={2},
   pages={215--242},
   issn={1083-4362}
}

\bib{Alt}{article}{
   author={Altmann, Klaus},
   title={Minkowski sums and homogeneous deformations of toric varieties},
   journal={Tohoku Math. J. (2)},
   volume={47},
   date={1995},
   number={2},
   pages={151--184},
   issn={0040-8735}
}

\bib{CLS}{book}{
   author={Cox, David A.},
   author={Little, John B.},
   author={Schenck, Henry K.},
   title={Toric varieties},
   series={Graduate Studies in Mathematics},
   volume={124},
   publisher={American Mathematical Society},
   place={Providence, RI},
   date={2011},
   pages={xxiv+841},
   isbn={978-0-8218-4819-7}
}

\bib{HS}{article}{
   author={Hausen, J{\"u}rgen},
   author={S{\"u}{\ss}, Hendrik},
   title={The Cox ring of an algebraic variety with torus action},
   journal={Adv. Math.},
   volume={225},
   date={2010},
   number={2},
   pages={977--1012},
   issn={0001-8708}
}

\bib{HW}{article}{
   author={Hausen, J{\"u}rgen},
   author={Wrobel, Milena},
    title = {Non-complete rational T-varieties of complexity one},
  journal = {ArXiv e-prints},
   eprint = {http://arxiv.org/pdf/1512.08930v1.pdf},
 keywords = {Mathematics - Algebraic Geometry, 14L30, 13A05, 13F15},
     year = {2015},
}

\bib{Ilt}{article}{
   author={Ilten, Nathan Owen},
   title={Deformations of smooth toric surfaces},
   journal={Manuscripta Math.},
   volume={134},
   date={2011},
   number={1-2},
   pages={123--137},
   issn={0025-2611}
}

\bib{IV}{article}{
   author={Ilten, Nathan Owen},
   author={Vollmert, Robert},
   title={Deformations of rational $T$-varieties},
   journal={J. Algebraic Geom.},
   volume={21},
   date={2012},
   number={3},
   pages={531--562},
   issn={1056-3911}
}

\bib{IV2}{article}{
   author={Ilten, Nathan Owen},
   author={Vollmert, Robert},
   title={Upgrading and downgrading torus actions},
   journal={J. Pure Appl. Algebra},
   volume={217},
   date={2013},
   number={9},
   pages={1583--1604},
   issn={0022-4049},
   review={\MR{3042622}},
   doi={10.1016/j.jpaa.2012.11.016},
}

\bib{Mav}{article}{
   author={Mavlyutov, Anvar},
    title = {Deformations of toric varieties
via Minkowski sum decompositions 
of polyhedral complexes
},
  journal = {ArXiv e-prints},
   eprint = {https://arxiv.org/pdf/0902.0967.pdf},
 keywords = {Mathematics - Algebraic Geometry},
     year = {2011},
}

\bib{Rei}{article}{
   author={Reid, Miles},
   title={Chapters on algebraic surfaces},
   conference={
      title={Complex algebraic geometry},
      address={Park City, UT},
      date={1993},
   },
   book={
      series={IAS/Park City Math. Ser.},
      volume={3},
      publisher={Amer. Math. Soc., Providence, RI},
   },
   date={1997},
   pages={3--159}
}

\bib{Ser}{book}{
   author={Sernesi, Edoardo},
   title={Deformations of algebraic schemes},
   series={Grundlehren der Mathematischen Wissenschaften [Fundamental
   Principles of Mathematical Sciences]},
   volume={334},
   publisher={Springer-Verlag, Berlin},
   date={2006},
   pages={xii+339},
   isbn={978-3-540-30608-5},
   isbn={3-540-30608-0}
}


\end{biblist}
\end{bibdiv}
\end{document}